\documentclass[12pt]{amsart}
\usepackage{extsizes}
\usepackage{amssymb,amsfonts,amsthm,amsmath}
\usepackage{graphicx}
\usepackage{color}
\usepackage[mathscr]{eucal} 
\usepackage{subcaption}
\usepackage{mathtools}
\usepackage[pagebackref=true]{hyperref}
\usepackage{hyperref}
    \hypersetup{
          unicode   = true,
          colorlinks= true,
    }
\usepackage{cite}

\usepackage{hyperref}
\usepackage{cleveref}
\crefformat{section}{\S#2#1#3} 
\crefformat{subsection}{\S#2#1#3}
\crefformat{subsubsection}{\S#2#1#3}

\usepackage[left=1 in,top=1 in,right=1 in, bottom=1 in]{geometry}
\usepackage[T1]{fontenc} 
\numberwithin{equation}{section}
\usepackage{bbm}

\definecolor{darkred}{rgb}{.70,.12,.20}

\definecolor{darkgreen}{rgb}{.20,.52,.14}

\definecolor{byz}{rgb}{.44,.16,.39}

\numberwithin{equation}{section}
\newtheorem{theorem}{Theorem}[section]
\newtheorem{definition}[theorem]{Definition}
\newtheorem{remark}[theorem]{Remark}

\newtheorem{lemma}[theorem]{Lemma}

\newtheorem{corollary}[theorem]{Corollary}
\newtheorem{hypotheses}[theorem]{Hypotheses}
\newtheorem{assumption}[theorem]{Assumption}

\newtheorem{Lad-Ur}{Ladyzhenskaya-Uraltceva iterative Lemma}
\usepackage{todonotes}

\usepackage{enumitem}

\title[Stability analysis of degenerate Einstein model of Brownian motion]{Stability analysis of degenerate Einstein model of Brownian motion}
\author{Isanka Garli Hevage$^1$, Akif Ibraguimov$^2$, Zeev Sobol$^{3}$ }
\date{}
  \usepackage[pagewise]{lineno}\linenumbers
\begin{document}
\nolinenumbers
\maketitle
\vspace{-1 cm}
\begin{center}
{$^1$ Department of Mathematics and Statistics, Sam Houston State University\\
Huntsville, Texas, USA, e-mail: \texttt{iug002@shsu.edu}
\smallskip
\\
{$^2$ Department of Mathematics and Statistics, Texas Tech University,}
\\
\small{Lubbock, Texas, USA, e-mail: \texttt{ilya1sergey@gmail.com}}
\smallskip
\\
{$^3$ Department of Mathematics, University of Swansea,}
\\
{Fabian Way, Swansea SA1 8EN, UK, e-mail: \texttt{z.sobol@swansea.ac.uk}}
}
\end{center}
\begin{abstract}
\noindent 
Our recent advancements in stochastic processes have illuminated a paradox associated with the Einstein-Brownian motion model. The model predicts an infinite propagation speed, conflicting with the second law of thermodynamics. The modified model successfully resolves the issue, establishing a finite propagation speed by introducing a concentration-dependent diffusion matrix. This paper outlines the necessary conditions for this property through a counter-example.
\\
The second part of the paper focuses on the stability analysis of the solution of the degenerate Einstein model. We introduce a functional dependence on the solution that satisfies a specific ordinary differential inequality. Our investigation explores the solution's dependence on the boundary and initial data of the original problem, demonstrating asymptotic stability under various conditions. These results have practical applications in understanding stochastic processes within bounded domains.
\end{abstract}
\section{Introduction}\label{intro}
In his well-known work \cite{Einstein05, Einstein56}, Einstein models the movement of a particle as a random walk, where the step time $\tau$ and the (random) displacement or \textit{free jump} $\Delta$ are both symmetrically distributed, independent of the point and time of observation. 
In this paper, the term \textit{free jump} is defined as the movement of particles without undergoing any collisions. It refers to the scenario where particles can move freely and independently, without encountering obstacles or interacting with other particles. The definition of \textit{free jump} in this context is based on the classical Einstein paradigm, as presented in his famous dissertation \cite{Einstein56}. It is worth noting that in the literature \cite{vin-krug}, the term \textit{free pass} is sometimes used interchangeably with \textit{free jump} to describe the same phenomenon. Therefore, in this context, \textit{free jump} and \textit{free pass} are synonymous terms.

The random walk is constrained by the mass conservation law, which is expressed through the concentration function. By utilizing Taylor's expansion, Einstein demonstrates that the concentration function $u$ which, represents the density of the particle distribution, satisfies the classical heat equation.
Although Einstein's paper was groundbreaking in the field of stochastic processes and an important step toward the construction of Brownian motion, the model he proposed led to a physical paradox. While $u$ is a solution to the heat equation, it allows a void volume to reach a positive concentration of particles instantly. Moreover, since the \textit{free jump}s process is reversible, the model permits all particles, with remarkable coherence, to concentrate instantly in a small volume. This contradicts the second law of thermodynamics and demonstrates an infinite propagation speed. 

By employing the De Giorgi--Ladyzhenskaya iteration procedure \cite{LU, DGV, ves-ted}, we successfully addressed the paradox in Einstein's original model. The resolution involved replacing the random walk with a diffusion process and allowing the diffusion coefficient to depend on the concentration function $u$, contrary to the constant coefficient assumption in Einstein's model (see \cite{IAS}). This modification aims to preserve the isotropic nature of the stationary media while removing the paradox by ensuring that the concentration function $u$ exhibits  \textit{finite propagation speed} property, namely,
if a neighborhood of a point is void of particles at time $t^*$, then a smaller neighborhood of the same point has been void of particles for some time before $t^*$.  Here, we provide a formal statement in our recent work in \cite{IAS}.
\begin{definition}[Finite propagation speed]\label{FPS}
A function $u\geq0$ on  $ \Omega \times [0,T)$ is said to exhibit a finite propagation speed if, for any open ball
$B\subset \Omega$ and any $\epsilon\in(0,1)$, there exists $T'\in (0,T]$ which might depend on $B$, $\epsilon$ and $u$, such that, given $u(x,0)=0$
for all $x\in B$, one has $u(x,t)=0$ for all $(x,t) \in \epsilon B\times[0,T')$.
\end{definition}
The concept of finite propagation speed was first demonstrated by G.I. Barenblatt \cite{Barenblatt-96} for a degenerate porous media equation, which is different from the Einstein model, as it is based on the traditional divergent equation for fluid density which is vanishing with pressure.
\\
This equation suggests that a finite propagation speed occurs when a small concentration leads to a small diffusion. To reflect the idea of higher medium resistance for small numbers of particles, we assume that the diffusion coefficient is a positive continuous function of concentration 
$u$ and that it degenerates when the concentration vanishes.
The concept of concentration functions finds frequent application in the analysis of stochastic processes and their associated partial differential equations. 

As mentioned in the prompt, concentration functions can be understood in two distinct manners: as a function or as a density of the distribution of particles. 
The first approach, where concentration functions are considered as a function, leads to the backward Kolmogorov equation. This equation describes the evolution of the probability distribution of a stochastic process from a given time to an earlier time. Specifically, the equation describes the evolution of the concentration function $u(x,t)$  at time $t$, given its value at a later time $T$ and position $x$. This equation is useful in analyzing the behavior of stochastic processes over time. 
The second approach leads to the forward Kolmogorov equation. This equation describes the evolution of the probability density function of a stochastic process from an initial time to a later time. Specifically, the equation describes the evolution of the concentration function $u(t,x)$  at time $t$, given its value at an earlier time $s$ and position $x$. This equation is useful in predicting the future behavior of stochastic processes.

We derive the generalized Einstein model by treating the concentration function $u$ as a function of both time and space \cite{Skorohod}.  For a space-time point of observation $Z=(x,s)\in \mathbb{R}^N\times[0,\infty)$, we consider an $\mathbb{R}^N$-valued random \textit{free-jumps process},
describing an interaction-free displacement of a particle off $Z$. Then,
Assuming the extended axioms as in \cite[Hypothesis 1]{IAS} and employing Taylor expansion in \cite{kon} to the generalized mass conservation law, we derive the governing partial differential equation (PDE) within a fixed domain $\Omega$ and over a time horizon $T>0$ for our scenario.
We introduce the upcoming hypothesis to delve into the core of the phenomenon, assuming neither drift nor consumption. The notion of a lower diffusion speed corresponding to a decreased concentration of particles is pivotal. This concept, crucial for achieving a finite propagation speed, suggests heightened medium resistance in scenarios with fewer particles.
\begin{hypotheses}\label{tau}
The diffusion matrix $a_{ij}(Z)=2a(u(Z))\delta_{ij}$, $i,j=1,2,\ldots,N$, for some scalar function $a\in C([0,\infty))$ with $a(0)=0$ and $a(s)>0$ for $s>0$. 
Furthermore, we assume that  \begin{equation}\label{I-s}
I(s) \triangleq \int_s^\infty \frac{d\tau}{\tau\,a(\tau)}
\end{equation}
is finite for all $s>0$ and
\begin{align}
 \limsup\limits_{s \to \infty}a(s)I(s)<\infty.\label{infty}
\end{align}
\end{hypotheses}
 Here $s \to a(s)$ is not necessarily differentiable. 
\begin{remark}
 $I(s) \to \infty$ as $s \to 0$. Moreover, 
from equations \eqref{infty} and \eqref{test1},  the function $s \to a(s)I(s)$ remains bounded and continuous. 
\end{remark}

Under Hypothesis \ref{tau}, our governing model is reformulated as the following inequality.
\begin{align}
    u_t \leq a(u) \Delta u, \quad \text{in } \Omega_T. \label{M-0}
\end{align}
 
To further analyze the above equation, we multiply \eqref{M-0} by a weight function $  h(v) > 0$, where $ h \in C((0,\infty)) \cap  \rm L_{loc}^{1}([0,\infty))$.
Let
\begin{align}
    H(v) \triangleq  \int_0^v { h(r) dr}, \text{ for } v \geq 0, \label{H-fun}
\end{align}
and introduce a locally Lipschitz continuous function
\begin{align}
   F(v) \triangleq h(v)a(v), \text{ for } v \geq 0,  \label{f-fun}
\end{align} which is nondecreasing. In addition require  $H(0)=0$ and $F(0)=0$. 

Consequently, we derive the following partial differential inequality 
\begin{align}
 [ H(u)]_t -   F(u)\Delta u \leq 0 \quad \mbox{in } \Omega_T. \label{div-e}
\end{align}
It is also important to assume that $F(v) \to 0$ when $v\to 0^{+}$. This regularization step allows us to analyze \eqref{div-e} and utilize a rich set of techniques for its weak solutions. 

In this context, we seek a solution to the partial differential inequality  \eqref{M-0} as a weak positive bounded solution to \eqref{div-e} satisfies the finite speed of propagation property. This property ensures that the solution does not propagate information faster than a certain speed, resolving the paradox associated with the violation of the second law of thermodynamics. The main result in \cite[Theorem 5.1]{IAS} establishes the existence of an \emph{a priori} finite propagation speed for a weak nonnegative
the bounded solution to the concentration equation \eqref{WEAK}. 
It states that the concentration $u$ demonstrates a finite propagation speed if the diffusion coefficient $a$ defined in  Hypothesis \eqref{tau} 
satisfies the following  two constraints
\begin{align}
    \limsup\limits_{s\to 0}a(s)I(s)<\infty,\label{test1} 
    \vspace{-4 cm}
    \end{align}
    and
    \begin{align}
     \exists \ c,\mu>0 \text{ such that } 
         a(s)I^\mu(s) \geq c \ a(v)I^\mu(v) , \text{ for } 0  < s < v < 1.  \label{test2}
\end{align}

These assumptions hold for  $a(s) = ks^\rho$, with $k$ and $\rho$ being positive constants as we proved in \cite{Isakif, ICA}. They also hold for more general cases, such as regularly varying functions with a strictly positive lower index. 
It is worth noting that we construct an example of a degenerate function $a(u)$ in such a way that a corresponding self-similar solution to the inequality
\begin{equation}\label{WEAK}
  \iint\limits_{\Omega_T}\nabla u \cdot \nabla(F(u)\phi)\,dxdt \leq 
  \iint\limits_{\Omega_T}H(u) \phi_t\,dxdt, \quad  \phi\in \rm Lip_c(\Omega_T),
\end{equation}
exhibits an infinite speed of propagation, while the constraint \eqref{test1} is not satisfied. Detailed examples offering further insight into the latter phenomena can be found in Section \ref{example-sec}.

In Section \ref{Stab-An}, we delve into the stability analysis of the degenerate Einstein model by studying the initial boundary value problem (IBVP) outlined in \eqref{ibvp-eq-1-1}-\eqref{ibvp-eq-1-2} with homogeneous boundary conditions. Our investigation spans both bounded and unbounded strong solutions, focusing on the central quantity of our analysis, $Y(t)$ in \eqref{y-est}. We will derive several estimates for this crucial quantity in various situations. We initiate our exploration by examining the fundamental degeneracy case of $u^{-\gamma}$, which not only provides insights but also establishes a foundation for our subsequent analysis.
To ensure the validity of our analysis, we impose essential assumptions on the functions $H$ and $F$ and demonstrate that the estimate $Y(t)$ remains bounded concerning the initial data. This serves as a fundamental basis for establishing the concept of asymptotic stability in forthcoming results.

Furthermore, we conduct a thorough analysis of the parameters within the framework of extended Assumption \ref{No-H'-2}. This analysis enables us to gain deeper insights into the stability conditions under different values of $\beta$. Furthermore, we expand Assumption \ref{No-H'-2} by introducing additional parameters, $\beta_1$ and $\beta_2$, thereby enhancing the flexibility and comprehensiveness of our results about asymptotic stability. The detailed findings and conclusions of this extension are presented in Theorem \ref{stab-latest}. Through these rigorous analyses and derived results, we aim to provide a comprehensive understanding of the stability aspects of the degenerate Einstein model. 

\section{Necessary Condition for finite speed of propagation}\label{example-sec}
In this section, we show that the condition \ref{test1} is essential for the property of finite speed of propagation in the sense of  Definition \ref{FPS}, namely, we will present the diffusion coefficient function $a$ that degenerates at $u=0$ but violates the condition \eqref{test1}, and has a solution to the corresponding inequality that exhibits the property of infinite propagation speed. 
We define
\begin{align}
Lu =      u_t   - & \ a(u)\Delta u. \label{ce-0}
\end{align}

Our objective is to identify a non-negative solution $u$ of the inequality   $Lu \leq 0$. This sup-solution possesses the following properties: it is positive for each moment of time $t > 0$ and each $x$, while being equal to zero for $t = 0$ outside of a compact set $K$. Assume  $u$ has the form
\begin{align}
     u(x,t)   = & \ \phi(t) f \left(\frac{|x|}{\theta(t)}\right), \text{ where } \theta(t) \neq 0. 
\end{align}
We compute 
    \begin{align}
    u_t =  &  \ \phi'(t) f(s) - \frac{\phi(t)\theta'(t)}{\theta(t)} sf'(s), \label{ce-1}\\
        \Delta u = & \ \frac{\phi(t)}{\theta^2(t)} \left[ f''(s) + \frac{(N-1)}{s}f'(s)\right], \text{ where } s={|x|}  \big/{\theta(t)} \label{ce-2}.
    \end{align}  
Set
 $\phi (t) = 1$ and  $\theta(t) =  \ \sqrt{2t}$. Then  ${\theta' (t)} \big/{\theta(t)} = {1} \big/{\theta^{2}(t)}$. 
Then \eqref{ce-0} becomes
\begin{align}
    -sf'(s) = & \ a(f(s))\left[ f''(s) + \frac{(N-1)}{s}f'(s)\right] =
    \ a(f(s)) \frac{1}{s^{N-1}}\left[s^{N-1}f'(s)\right]', 
   \end{align} 
which yields
\begin{align}
      -\frac{s}{a(f(s))} =   \frac{d}{ds} \Big ( \ln  |s^{N-1}f'(s)| \Big). \label{ce-7}
   \end{align}
Let  $f$ be the form
\begin{align}
    f(s) = & \ e^{\displaystyle -s^{\lambda}} \text{ for } s > 0,  \lambda >2.  \label{ce-6} 
\end{align}
Then, one can define the  function $a: [0, \infty) \to [0,\infty)$ by
\begin{align}\label{a-4}
   a(u) = 0 \mathbbm{1}_0 +   \frac {|\ln u|^{\frac{2}{\lambda}}}{  \lambda|\ln u|-(\lambda+N-2)} \mathbbm{1}_{(0,u_{*})} 
   + {a_*}{(u+1)} \mathbbm{1}_{[u_*, \infty)},
\end{align}
for constants $u_* \in (0,1)$ and $a_* > 0$, such that $a \in C([0, \infty))$ and $a(s)>0$ for $s > 0$. 

\vspace{0.2 cm}
\noindent Note that
\begin{align} \label{c-star}
c_* \triangleq \int_{u_*}^{\infty} \frac{d\tau}{\tau a(\tau)} < \infty.
\end{align}
   Also, there is $C_1>0$ such that 
\begin{align}
a(u)  \geq  C_1 {|\ln u|^{\frac{2}{\lambda}-1}} \text{ for } u \in (0,u_*]. \label{a-2}    
\end{align}  
By \eqref{a-4} and using \eqref{c-star},
we obtain
   \begin{align}
       I(u) 
          = & \frac{\lambda}{2-\frac{2}{\lambda}}  \left(|\ln (u)|^{  2-\frac{2}{\lambda}}-  |\ln (u_*)|^{  2-\frac{2}{\lambda}} \right) \nonumber\\
           &  \hspace{4.1 cm}- 
        \frac{\lambda+N-2}{1-\frac{2}{\lambda}} \left(|\ln (u)|^{  1-\frac{2}{\lambda}}-|\ln (u_*)|^{  1-\frac{2}{\lambda}}\right) + c_{*},  \label{ce-10}
 \end{align}
for $u \in (0, u_*]$.  Combining  \eqref{ce-10} and \eqref{a-2},  we get
\begin{align}
    a(u) I(u) \geq C_2 |\ln (u)|^{\frac{2}{\lambda}-1+2-\frac{2}{\lambda}}= C_2|\ln (u)| \text{ for all } u \in (0, u_*].
\end{align}
Therefore,
 $ a(u) I(u) \to \infty  \text{ when } u \rightarrow 0. $ 
Hence, $a(u)$ does   not satisfy  \eqref{test1}.
\vspace{0.2 cm}

For the condition 
 $ \displaystyle s = \frac{|x|}{\sqrt{2t}} > s_{0},$  we particularly choose $   |x| > \sqrt{2}s_{0} \text{ and } t < 1. $ 
\\ With such $(x,t)$, we have $u(x,t)= e^{-\left(\frac{|x|}{\sqrt{2t}}\right)^{\lambda}} < u_*$. 

\vspace{0.2 cm}In summary,  $u(x,t)= e^{-\left(\frac{|x|}{\sqrt{2t}}\right)^{\lambda}}$,  with $\lambda >2$, satisfies equation
$$u_t = a(u)\Delta u \text{ for } 0 <t<1, |x| > \sqrt{2}s_0,$$ 
the function $a(u)$ given by  \eqref{a-4} does   not satisfy  \eqref{test1}, and,  for all $x\neq 0$, 
$$ \lim_{t \to 0^{+}} u(x,t)=0, \ u(x,t) > 0 \text{ for any } t>0, $$

i.e., $u(x,t)$ has infinite speed of propagation. 

\section{Stability analysis of degenerate Einstein model}\label{Stab-An}

We will start with a simple observation for a basic degenerate model with degeneration of the power degree, and define the solution of the problem. In the next section, we will generalize this observation for more complex non-linearity.
In the stability analysis of the degenerate Einstein model under homogeneous boundary conditions, our emphasis is on the central quantity
\begin{align} \label{y-est}
Y(t) = \int_{\Omega}[H(u(x, t))]^{\nu} dx, \text{ with } \nu \geq 1,
\end{align}
for $t \geq t_0$. To ensure coherence in our work, we consistently apply the following fundamental estimate throughout this section.
\begin{lemma} \label {Aux-lem}
    Let $ k > 0$ and $\beta > 1 $ be constants. Let $Y(t) \in AC([t_0, \infty)) $ be a nonnegative  solution of
\begin{align}\label{ODI-st}
    \frac{d}{dt} Y(t)  +  k \  [Y(t)] ^{\beta}\ \leq 0,  \text{ for }  
     t >t_0, 
\end{align}
with $ Y(t_0) \triangleq Y_0$. 
\begin{enumerate}
    \item [\rm (i)] If $Y_0 = 0$ then $Y(t) = 0$ for all $t\geq t_0$.
    \item [\rm (ii)] If $Y_0 > 0$ then 
   \begin{align}\label{beta > 1}
         Y(t) \leq  {\Big[  k(\beta-1)(t-t_0) +{Y_0^{-(\beta-1)} } \Big]^{-\frac{1}{\beta-1}}} \text{ for }  t \geq t_0. 
   \end{align}
Consequently,
\begin{align}
    \lim_{t \to \infty} Y(t) = 0.
\end{align}
\end{enumerate}
\end{lemma}
\begin{proof}
    {Part (i).} Using the inequality \eqref{ODI-st}, we get
\begin{align}
      \frac{d}{dt}[Y(t)] \leq \ & -k [Y(t)]^{\beta}. 
\end{align}
Then $ \displaystyle
        \frac{d}{dt}[Y(t)] \leq 0.
$ Consequently, if $Y_{0} = 0$ then $Y(t) = 0 $ for all $t\geq t_0$. 
\vspace{0.4 cm}\\
\noindent {Part (ii).}  Let $Y_{0} >0$.
\begin{itemize}
     \item 
 Case 1:  $Y(t) > 0$ for all $t \geq t_0$.  Using the inequality \eqref{ODI-st}, we get
\begin{align}
\frac{d}{dt} [Y(t)]^{-\beta+1} 
    \ & \leq  - k (1-\beta),
\end{align}
Integrating both sides over $(t_0, t)$, yields,
\begin{align}
     {[Y(t)]^{-\beta +1}} - {Y_0^{-\beta  +1}}  & \ \geq - k  (1-\beta)(t-t_0), \\
 [Y(t)]^{-\beta+1} \ &  \geq k(\beta-1)(t-t_0) + Y_0^{-\beta+1}. 
\end{align}
From above \eqref{beta > 1} follows.
\item Case 2:  There is $t^*>t_0$ such that $Y(t^{*})=0$. Then there exists ${t'} \in (t_0,t^*]$ so that $Y(t') = 0$ and $Y(t)>0$ for all $t\in [t_0,t')$ and $Y(t)=0$ for all $t \geq t'$. Hence, the inequality \eqref{beta > 1} holds for all for all $t \geq t_0$.
\end{itemize}

Consequently, in both cases, the inequality \eqref{beta > 1} holds for all $t \geq t_0$. 
\end{proof}

\begin{remark}
    The result of the Lemma under some assumption on speed of convergence of function $\mathcal{Y}(s)\to \infty$, as $s\to \infty $  one can  consider  more general ordinary differential inequality 
    \begin{align}\label{ODI-s-gen}
    \frac{d}{dt}Y(t)+ \mathcal{Y}\left(Y(t)\right) \leq 0,  \text{ for }  
     t >t_0, 
\end{align}

  This will make it possible to consider the stability of a more sophisticated non-linear Einstein model.    
\end{remark}

\subsection{Basic degenerate model}
    
As a model let us start with the basic non-linear function  $a(u) = K u^{\gamma}$, for constants $K>0$ and $\gamma \in (0,1)$.
Suppose $ u\geq0$ is a solution of the following IBVP
\begin{align} 
(u^{1-\gamma})_t  - K \Delta u   \leq  \ &  \ 0   \text{ in }  \Omega\times (0,\infty),\label{ibvp-basic-1}\\
u(x,0)  =   \ &   u_0(x)      \text{ on }  \Omega,   \\
u(x,t)  =   \ &  0    \text{ on } {\partial \Omega \times (0,\infty)}, \label{ibvp-basic-3}
\end{align}
where $K>0$ is a constant.

We obtain the following asymptotic stability result for the solution of the above IBVP in the following sense.
\begin{definition} \label{sol-basic}
We call  $u(x,t)$ is solution of the problem \eqref{ibvp-basic-1}-\eqref{ibvp-basic-3} if $(u^{1-\gamma})_t \in \rm L^2(\Omega\times(0,\infty))$, $\Delta u \in \rm L^2(\Omega)$, and  the divergent formula is applicable for the vector field $\nabla u$, such that inequality \eqref{ibvp-basic-1} holds alomat everywhere.
\end{definition}

\begin{theorem} \label{ex-1} 
Assume  $u \geq 0$ and be defined as in Definition \ref{sol-basic}.  Let  $m \geq 2-\gamma$. Define
\begin{align}
    Y(t) \triangleq \int_{\Omega} [u(x,t)]^{m} dx, \text{ for }\ t \geq t_0, \text{ and }   Y(t_0) \triangleq Y_{0}.
\end{align}
\noindent Then there exists a constant $k_1>0$ independent of the solution $u(x,t)$ such that  
\begin{align}
 Y'(t) + k_1 [Y(t)]^{\frac{m+\gamma}{m}} \leq 0, \text{ for all } t> t_0.
\end{align}
If $Y_{0} =0$ then  $Y(t) = 0$ for all $t\geq t_0$. 
If  $Y_{0} > 0$ then 
 \begin{align}\label{bacis--yt-est}
         Y(t) \leq  {\Big [\frac{k_1\gamma}{m} (t-t_0) +{Y_0^{-{\frac{\gamma}{m}}} } \Big]^{-\frac{m}{\gamma}}} \text{ for all }  t \geq t_0, 
   \end{align}and consequently,
$$ \lim_{t \rightarrow \infty} Y(t) = 0.$$ 

\end{theorem}
\begin{proof}
First of all,  note that $ m+\gamma-1 \geq 1$ and 
\begin{align*}
 (u^{1-\gamma})_t \cdot u^{m+\gamma-1} 
 =  
   \big (u^{1-\gamma} \big)_t \cdot \big(u^{1-\gamma}\big)^{\frac{m}{1-\gamma}-1} 
    =  
    \frac{1-\gamma}{m}  \big(u^{m}\big)_t.
\end{align*}
Multiplying  \eqref{ibvp-basic-1} by $u^{m+\gamma-1}$, and using integration by parts over $\Omega$, we find that
\begin{align}
  \frac{1-\gamma}{m} \cdot \frac{d}{dt} \int_{\Omega} u^{m} \ dx 
    + K  (m+\gamma-1) \int_{\Omega} u^{m+\gamma-2} |\nabla u|^{2}  \ dx \leq 0. \label{b-1}
\end{align}
We compute \[ \displaystyle
    \int_{\Omega} u^{m+\gamma-2} |\nabla u|^{2} dx  = \frac{4}{(m+\gamma)^2} \int_{\Omega} \big \vert \nabla u^ {\frac{m+\gamma}{2}} \big \vert ^2 \ dx \ .     \vspace{0.2 cm} \] 
Then \eqref{b-1} becomes
\begin{align}
  \frac{d}{dt}  \int_{\Omega} u^{m} \ dx 
    + 
    \frac{4Km(m+\gamma-1)}{(1-\gamma)(m+\gamma)^{2}} \int_{\Omega} \big \vert \nabla u^ {\frac{m+\gamma}{2}} \big \vert ^2 \ dx \leq 0 \ . \label{b-2}
\end{align}
Let 
\begin{align}
    q= {2m} \big / ({m+\gamma}). \label{q-def-1}
\end{align}
Note that  $q \in (1,2)$. By 
 Poincar\'e-Sobolev inequality \cite{EVAN-1} with power $q$ in \eqref{q-def-1}, we obtain
\begin{align}
\displaystyle
  \left [ \int_{\Omega} \Big(u^{\frac{m+\gamma}{{2}}}\Big)^{{q}} \ dx \right ]^ {\frac{2}{{q}}} 
  \leq \ & c_0^{2} \int_{\Omega} \big| \nabla \displaystyle u^{\frac{m+\gamma}{{2}}}    \big|^2 \ dx,
  \end{align}
where $c_0$ is a positive constant that depends on $N$.  Hence,
\begin{align}
 c_0^{-2} [Y(t)]^{\frac{m+\gamma}{m}} \leq \ &  \int_{\Omega} \big| \nabla \displaystyle u^{\frac{m+\gamma}{{2}}}    \big|^2 \ dx, \label{b-7.5}
\end{align}
Using \eqref{b-7.5} in \eqref{b-2}, we obtain
\begin{align}
   \frac{d}{dt} Y(t)
 \   + \ \frac{4Km (m+ \gamma - 1)}{c_0^2 (1-\gamma)(m+\gamma)^2} \  [Y(t)] ^{\frac{m+\gamma}{m}}
    \ \leq  0 \ .
    \label{b-odi}
\end{align}
Note that $\frac{m+\gamma}{m} > 1$. By the virtue of  Lemma \ref{Aux-lem}, Theorem \ref{ex-1} holds.
\end{proof}
\begin{remark}
    One can easily consider the case when $\gamma=0$ in Theorem  \ref{ex-1} and obtain the asymptotic exponential stability of the classical heat equation.
\end{remark}
Next, we will investigate the stability of the generalized Einstein model with the inequality \eqref{div-e}.
\subsection{Generalized degenerate Model}

\begin{align} 
[H(u)]_t -   F(u)\Delta u  &  \leq 0  \text{ in }  \Omega\times ( t_0,\infty), \label {ibvp-eq-1-1}
 \\
 u(x,t_0)   & =  u_0(x)    \mbox{ on }  \Omega,   \\
  u(x,t)  & =  0\mbox{ on } \ {\partial \Omega \times ( t_0,\infty)}.  \label {ibvp-eq-1-2}
 \end{align}

Recall that $H$ and $F$ are defined by equations \eqref{H-fun} and \eqref{f-fun}, respectively, with the conditions $H(0)=0$ and $F(0)=0$.

\subsubsection{Stability analysis of bounded  solutions}
We study the stability of the bounded solution of the above IBVP in the following sense.
\begin{definition} \label{bdd-sol-1}
We call $u(x,t)$ is a bounded solution of the
IBVP \eqref{ibvp-eq-1-1}-\eqref{ibvp-eq-1-2} if
\begin{align}\label{sup-m}
     \sup_{\Omega \times (0,\infty)}u(x,t) \triangleq M < \infty,
\end{align} 
   $[H(u)]_t \in L^2(\Omega \times (0, \infty))$ and
    $F(u) \Delta u \in  L^2(\Omega)$,  such that \eqref{ibvp-eq-1-1} holds almost everywhere.
\end{definition}
\noindent Then we assume the following for the functions $H$ and $F$.
\begin{assumption}\label{HH'-product}
There exists  $\gamma_1>0$ such that  
$H^{\gamma_1+1} \in C^{1}([0,\infty))$, and   for all $ M >0$,  there is $c_{1} = c_1(M) >0 $ such that
\begin{align}\label{HH'-product-R1}
  [H(s)^{\gamma_1+1}]' \leq c_1  , \text{ for all }  \ s\in [0,M] . 
 \end{align} 
\end{assumption}
It follows from \eqref{HH'-product-R1} that 
$$ 
     H(s) \leq  \ c_1^{\frac{1}{\gamma_1+1}} s^{\frac{1}{\gamma_1+1}}.$$
     Evidently by Definition \ref{H-fun}, we get 
     \begin{align}\label{H-prim}
           [H(s)^{\gamma_1+1}]'\geq 0.
     \end{align}
\begin{assumption} \label{FH-product-1}
     For any $ M >0$, there exists $c_2=c_2(M) > 0  $ such that
    $$[F(s)H^{\gamma_1+1}(s)]'\geq c_2, \text{ for all } s \in [0, M].$$
 \label{c-1-0}
\end{assumption}
\vspace{-1 cm}
\noindent Based on the above assumptions, we now state the following theorem on asymptotic stability.
\begin{theorem} \label{stab-s-0}
Let  Assumption {\ref{HH'-product}}, and Assumption {\ref{FH-product-1}} hold. Assume  $u \geq 0$ is a bounded solution for IBVP \eqref{ibvp-eq-1-1}-\eqref{ibvp-eq-1-2}. Let
$p \geq \gamma_1+1$.  Define
\begin{align}
     Y(t) \triangleq  \int_{\Omega} [H(u (x,t))]^{p+1} \ dx, \text{ for } t \geq t_0, \text{ and } Y_{0} \triangleq Y(t_0).
\end{align}
Then $Y(t)$ satisfies the differential inequality 
\begin{align}
\frac{d}{d t} Y(t)
+ c [Y(t)]^{\frac{p+1+\gamma_1}{p+1}} \leq 0, \text{ for all } t > t_0,
\label{s-diff-0}
\end{align}
where $ c > 0 $ is a constant depending on $M$.  If $Y_{0} =0$ then  $Y(t) = 0$ for all $t\geq t_0$. 
If  $Y_{0} > 0$ then 
\begin{align}
        Y(t) \leq  {\Big[  c\left(\frac{\gamma_1}{p+1}\right)(t-t_0) +{Y_0^{-\left(\frac{\gamma_1}{p+1}\right)} } \Big]^{-\frac{p+1}{\gamma_1}}} \text{ for }  t \geq t_0, 
\end{align}
and consequently,
\begin{align}
   \lim_{t \rightarrow \infty} Y(t) =0.   \label{y-lim}
\end{align}
\end{theorem}
\begin{proof}
With $M$ defined by \eqref {sup-m}, let $c_1= c_1(M)$ and  $c_2= c_2(M)$ be positive constants in Assumption \ref{HH'-product} and  Assumption \ref{FH-product-1} respectively.

Note that $[H(u)]^{p} F(u) \big \vert_{\partial \Omega} = 0$. Multiply the  first inequality in \eqref{ibvp-eq-1-1} by $[H(u)]^{p}$, and integrating over $\Omega$, we obtain
\begin{align}
    \int_{\Omega} [H(u)]^{p} [H(u)]_t \ dx 
    -
    \int_{\Omega}  [H(u)]^{p} F(u) \Delta u  \ dx   \ \leq 0.  \label{s-0-0-0}
\end{align}
Using integration by parts for the second integral in \eqref{s-0-0-0}, we find
\begin{align}
  \frac{d}{dt}   \int_{\Omega} [H(u)]^{p+1} \ dx + (p+1) \int_{\Omega} \nabla u \cdot \nabla  ( [H(u)]^{p} F(u) ) \ dx   \ \leq 0.  \label{s-1-0-0-1}
\end{align}
Note that 
\begin{align}
    [H^{p}]'  =  \left([H^{\gamma_1+1}]^{\frac{p}{\gamma_1+1}} \right)'
              = \frac{p}{\gamma_1+1} [H^{\gamma_1+1}]^{\frac{p}{\gamma_1+1}-1} [H^{\gamma_1+1}]'
             = \frac{p}{\gamma_1+1} H^{p-(\gamma_1+1)} [H^{\gamma_1+1}]'.
\end{align}
Recall that $p \geq \gamma_1 +1$, and $H$ subject to \eqref{H-prim}. Using the above, we  show that
\begin{align}
\nabla u \cdot \nabla (  [H(u)]^{p} F(u) ) & =  
  \vert  \nabla u \vert^2
  F(u)   \frac{p}{\gamma_1+1} H^{p-(\gamma_1+1)} [H^{\gamma_1+1}]'  +   \ \vert \nabla u \vert^2 [H(u)]^{p} F'(u)   \\
 & =  
\left[ \frac{\gamma_1+1}{p} F'(u)H^{\gamma_1+1}(u) + F(u) [H^{\gamma_1+1}]'(u)\right]  \cdot \frac{p}{\gamma_1+1}\vert  \nabla u \vert^2[H(u)]^{p-(\gamma_1+1)}\\
  & \geq  [F H^{\gamma_1+1} ]'(u) \ \vert  \nabla u \vert^2[H(u)]^{p-(\gamma_1+1)}.
\end{align}
By  Assumption  \ref{FH-product-1}, we obtain 
\begin{align}
\nabla u \cdot \nabla (  [H(u)]^{p} F(u) ) \geq
  \ c_2 \cdot \vert  \nabla u \vert^2[H(u)]^{p-(\gamma_1+1)}.
\label{s-2-0-0}
\end{align}
Using \eqref{s-2-0-0} in the inequality \eqref{s-1-0-0-1} becomes
\begin{align}
 \frac{d}{dt}   \int_{\Omega} [H(u)]^{p+1} \ dx
+
c_2 (p+1) \int_{\Omega} \vert  \nabla u \vert^2 [H(u)]^{p-(\gamma_1+1)} \ dx  \ \leq  \ 0 \ .  \label{s-4-0-0}
\end{align}

Let 
\begin{align}
    q = {(2p+2)} \big / ({p+\gamma_1+1}). \label{q-4.26}
\end{align}
Note that  $ q \in(1,2) $. Applying Poincar\'e-Sobolev inequality with power $q$ in \eqref{q-4.26}, we obtain
\begin{align}
    \Bigg[\int_{\Omega} \left( \left[H(u)^{\gamma_1 +1}\right]^{\frac{\gamma_1}{(\gamma_1+1)(2-q)}} \right)^{q} dx \ \Bigg]^{\frac{2}{q}} 
 & \leq 
  c_0^2  \int_{\Omega} \big | \nabla \left[ H(u)^{\gamma_1 +1} \right]^{\frac{\gamma_1}{(\gamma_1+1)(2-q)}} \big |^{2} dx,
 \label{s-2-0-0-11}
\end{align}
where $c_0>0$ represents a constant that depends on $N$.
Observe  that 
$$ \left( \left[ H^{\gamma_1 + 1}\right]^{\frac{\gamma_1}{(\gamma_1+1)(2-q)}}\right)^{q} =  H^{p+1}.$$ Therefore, \eqref{s-2-0-0-11} becomes
\begin{align}
 [Y(t)]^{\frac{2}{q}} = \left[\int_{\Omega} [H(u)]^{p+1} \ dx \right]^{\frac{2}{q}}    & \leq 
  K_1  \int_{\Omega}  \left |       \left[  H(u)^{\gamma_1 +1} \right]^{\frac{\gamma_1}{(\gamma_1+1)(2-q)}-1} ( H(s)^{\gamma_1+1})'\vert_{s=u(x,t)}  \nabla u \right |^{2}  dx, \label{s-2-0-0-1155}
\end{align}
for $$
    K_1 =  \frac{c_0^2 \gamma_1^2} {(\gamma_1+1)^2(2-q)^2}.
 $$
Using Assumption \ref{HH'-product} in \eqref{s-2-0-0-1155}, gives 
\begin{align}
[Y(t)]^{\frac{2}{q}} 
 \leq 
   c_1 K_1   \int_{\Omega}  \big |       \left[  H(u)^{\gamma_1 +1} \right]^{\frac{\gamma_1}{(\gamma_1+1)(2-q)}-1} \nabla u \big |^{2}  dx.\label{s-2-0-0-1157}
\end{align}
We compute the exponent on the right-hand side of \eqref{s-2-0-0-1157} to obtain
\begin{align}
{\frac{\gamma_1}{(\gamma_1+1)(2-q)}-1}= \frac{p-(\gamma_1+1)}{2 (\gamma_1+1)}.
\end{align}
Hence,
\begin{align}
[Y(t)]^{\frac{2}{q}}
\leq 
c_1 K_1
  \int_{\Omega} \left[ H(u)\right]^{p-(\gamma_1+1)} |\nabla u|^{2} \ dx. \label{s-2-0-0-12}
\end{align}
Using \eqref{s-2-0-0-12} in \eqref{s-4-0-0}, we obtain
\begin{align*}
[Y(t)]'
+
\frac{c_2 (p+1)}{c_1K_1} [Y(t)]^{\frac{2}{q}}  \ \leq  \ 0,  \label{s-4-0-0-14}
\end{align*}
which proves the inequality \eqref{s-diff-0}. Consequently, we obtain the limit \eqref{y-lim}.
\end{proof}

\subsubsection{Stability analysis of unbounded  solutions}

\begin{definition}\label{unbddsol1}
Let $u$ be a  non-negative solution of the IBVP \eqref{ibvp-eq-1-1}-\eqref{ibvp-eq-1-2} provided that 
   $[H(u)]_t \in L^2(\Omega \times (0, \infty))$ and 
    $F(u) \Delta u \in  L^2(\Omega \times (0, \infty)) $, such that \eqref{ibvp-eq-1-1} holds almost everywhere.
\end{definition}

In this subsection, we study unbounded solution in Definition \ref{unbddsol1} by introducing the following relaxed version of the Assumption  \ref{FH-product-1}.
\begin{assumption}\label{unbb-assump}
  Let $p\geq 0$ such that  $[F(s)H^{p}(s)]'\geq 0 $  for all $s \in [0, \infty)$. 
\end{assumption} 
\begin{assumption} \label{w-12}
   Assume $u$ is such that $[H(u)]^pF(u) \in W^{1,2}_{0} (\Omega)$, for  $p$ in Assumption \eqref{unbb-assump}.
\end{assumption}


Next, we establish a primary property of the solution $u(x,t)$.
\begin{theorem} \label{H-bound}
 Let $p \geq 0$ and Assumptions \ref{unbb-assump} and \ref{w-12} hold. Assume  $u \geq 0$ is a  solution of  IBVP  \eqref{ibvp-eq-1-1}-\eqref{ibvp-eq-1-2}. Suppose
 $H(u_0(x)) \in \rm L^{p+1}(\Omega)$. Define
     \begin{equation}\label{Y-p-0}
     Y(t)= \int_{\Omega} [H(u (x,t))]^{p+1} \ dx, \text{ for } t \geq t_0, \text{ and } Y_{0} \triangleq Y(t_0).     
     \end{equation}
  Then $Y(t)$ is nonincreasing (monotone), and 
     \begin{equation}
    \label{s-1-0-1}  
    \int_{\Omega} [H(u (x,t))]^{p+1} \ dx    \leq   \ \int_{\Omega} [H(u_0 (x))]^{p+1} \ dx, \text{ for all } t \ge t_0. 
   \end{equation}

\end{theorem}

\begin{proof}
Multiply first inequality in \eqref{ibvp-eq-1-1} by $[H(u)]^{p}$, and integrating over $\Omega$,
\begin{align}
 \frac{1}{p+1} \frac{d}{dt}   \int_{\Omega} [H(u)]^{p+1} \ dx +  \int_{\Omega} \nabla u \cdot \nabla  \big ( [H(u)]^{p} F(u) \big) \ dx   \ \leq 0.  \label{s-1-0-0}
\end{align}
By Assumption \ref{unbb-assump}, we compute
\begin{align} \label{H^p-prop}
    \nabla u \cdot \nabla \big (  [H(u)]^{p} F(u) \big) 
    = & \
  \ \vert  \nabla u \vert^2
[H^{p}(s)F(s)]'\bigg \vert_{s=u}
 \geq  0. 
\end{align}
Using \eqref{H^p-prop} in \eqref{s-1-0-0} we get
\begin{equation}\label{mon-Y}
 \displaystyle
 \frac{d}{dt}  \left[ \int_{\Omega} [H(u)]^{p+1} \ dx   \right] =\displaystyle
 \frac{d}{dt} Y(t)  \ \leq  \ 0.
\end{equation}
From the above, the monotonicity of $Y(t)$ follows. Consequently
\begin{equation*}
    Y(t)\leq Y(t_0) \ \textnormal{for} \ t\geq t_0. 
\end{equation*}
With the above, we conclude inequality in \eqref{s-1-0-1}.
\end{proof}
Moving forward, we will continue our analysis with the functions $u, H$ and $u_0$ as stipulated in Theorem \ref{H-bound}.
 Additionally, we will introduce the following two structural conditions on the functions $H$ and $F$ before delving into our next results.

\begin{assumption} \label{No-H'-1} 
    There exist  $ p_1 >0, \ q_1 > 0$ and $c_3 > 0$ such that
    \begin{align}
        \big[H^{p_1}(s) F(s)\big]'\geq  c_3 [H(s)]^{q_1}, \ \forall  \ s \in [0,\infty).
    \end{align}
\end{assumption}

\begin{assumption} \label{No-H'-2}
 There exist $\gamma_1, \beta > 0$ and $c_4 > 0$ such that
\begin{align}
    [H^{\gamma_1+1}(s)\big]' \leq  c_4 [H(s)]^{\beta} \ \forall  \ s \in [0,\infty).
\end{align}
\end{assumption}

\begin{remark}
    Example for assumptions  \ref{No-H'-1} and \ref{No-H'-2}
\end{remark}
Given the above assumptions, we first state the following theorem when $\beta =  \dfrac{q_1}{2}$.
\begin{theorem}  \label{TH-2}
Let  Assumption  \ref{No-H'-1} and Assumption \ref{No-H'-2} hold for $\beta = q_1\big / {2}$. Assume
 $H(u_0(x)) \in \rm L^{p_1+1}(\Omega)$. We define $ \displaystyle
     Y(t) \triangleq  \int_{\Omega} [H(u (x,t))]^{p_1+1} \ dx,  t \geq t_0.
$ 
Let\begin{align}  \label{delta-def-1}
    \delta_1 ={(p_1+1)} \big /{(\gamma_1+1)}.
\end{align}
Assume $1 \leq \delta_1 < 2$. 
Then there exists $ c > 0 $ such that 
\begin{align}
\frac{d}{d t} [Y(t)] 
+ c[Y(t)]^{\frac{2}{\delta_1}} \leq 0, \text{ for all }  t > t_0.
\label{s-diff-0}
\end{align}
If $Y_0 >0$, then
\begin{align}
      Y(t) \leq  {\Big[  c \left(\frac{2-\delta_1}{\delta_1}\right)(t-t_0) +{Y_0^{-\left(\frac{2-\delta_1}{\delta_1}\right)} } \Big]^{-\left(\frac{\delta_1}{2-\delta_1}\right)}} \text{ for }  t \geq t_0,
\end{align}
and consequently,
\begin{align}
    \lim_{t \rightarrow \infty} Y(t) =0. \label {y-lim-1-2}
\end{align}

\end{theorem}
\begin{proof}
Note that  $[H(u)]^{p_1} F(u) \big \vert_{\partial \Omega} = 0$. Multiplying the first inequality in \eqref{ibvp-eq-1-1} by $[H(u)]^{p_1}$, and using integration by parts, we obtain
\begin{align}
  \frac{d}{dt}   \int_{\Omega} [H(u)]^{p_1+1} \ dx + (p+1) \int_{\Omega} \nabla u \cdot \nabla  ( [H(u)]^{p_1} F(u) ) \ dx   \ \leq 0.  \label{No-H'-12}
\end{align}
By Assumption \ref{No-H'-1}, we get
\begin{align}
\nabla u \cdot \nabla (  [H(u)]^{p_1} F(u) ) = 
 \
[H^{p_1}(u) F(u)]'\vert  \nabla u \vert^2 
 \geq   \ 
  \ c_3 \cdot [H(u)]^{q_1} \vert  \nabla u \vert^2.
\label{No-H'-13}
\end{align}
Using \eqref{No-H'-13} in \eqref{No-H'-12} becomes
\begin{align}
 \frac{d}{dt}  \left[  \int_{\Omega} [H(u)]^{p_1+1} \ dx \right]
+
c_3 (p+1) \int_{\Omega}  [H(u)]^{q_1} \vert  \nabla u \vert^2 \ dx  \ \leq  \ 0.  \label{No-H'-14}
\end{align}
By Poincar\'e-Sobolev inequality for $1 \leq \delta_1 < 2$, we get
 \begin{align}
[Y(t)]^{\frac{1}{\delta_1}} = \left[ \int_{\Omega} \left([H(u)]^{ \gamma_1+1 }\right)^{\delta_1} \ dx \right]^{\frac{1}{\delta_1}} 
     \leq   \ &
    c_p \left[\int_{\Omega} \bigg |  \nabla  \left([H(u)]^{  \gamma_1+1} \right)\bigg |^{2} \ dx \right]^{\frac{1}{2}},
     \end{align} 
     for some constant $c_p >0$, depending on  $N$. Hence,
     \begin{align}
         [ Y(t)]^{\frac{2}{\delta_1}}
       \leq   \ &
        c_p^{2} \int_{\Omega} \bigg |    \frac{d}{ds} \left(H(s)^{ \gamma_1+1} \right) \bigg |_{s=u(x,t)}^{2}  \vert \nabla u\vert ^{2} \ dx. \label{No-H'-15}
     \end{align}
By  Assumption \ref{No-H'-2}, we obtain from \eqref{No-H'-15}, that
\begin{align}
  [Y(t)]^{\frac{2}{\delta_1}}  \leq   
        c_4 \cdot c_p^{2}   \int_{\Omega} [H(s)]^{2\beta} \vert \nabla u\vert ^{2}\ dx  
        =   c_4 \cdot c_p^{2}   \int_{\Omega} [H(s)]^{q_1} \vert \nabla u\vert ^{2}\ dx.
        \label{No-H'-16}
\end{align}
Utilizing inequality  \eqref{No-H'-16} in \eqref{No-H'-14} provides
\begin{align}
   Y'(t) 
+
[p_1+1] \frac{c_3}{c_4 \cdot c_p^{2}}  [Y(t)]^{\frac{2}{\delta_1}}   \leq  \ 0, \text{ for all } t>0,\label{No-H'-17}
\end{align}
which implies inequality \eqref{s-diff-0}. Here  $ \delta_1 < 2$. Thus, by Lemma \ref{Aux-lem}, we obtain the estimate, and consequently, we derive the limit \eqref{y-lim-1-2}
\end{proof}

We establish the stability result for $\beta > \frac{q}{2}$, starting with 
the next auxiliary result.
\begin{theorem} \label{prop-1}
    
Let  Assumptions  \ref{No-H'-1} and \ref{No-H'-2} hold for $ \beta > q_1 \big / 2 $.
Let $\delta_1$ be defined as   in \eqref {delta-def-1}, where $1 \leq \delta_1 < 2$.
Moreover, let 
\begin{align} \label{delta-2-1}
\delta_2=  & 
    1,  \text{ if } 1 \leq \delta_1 \le \frac{N}{N-1},
    \end{align}
   or
    \begin{align}\label{delta-2-2}
 \delta_2= & \frac{\delta_1N}{N+\delta_1},   \text{ if } \frac{N}{N-1} < \delta_1 < 2.
\end{align}
Assume that for any  $\infty>t \geq t_0$. Then there exists  $ k > 0$ such that
\begin{align}\label{y-1}
Y'(t) 
+
\frac{k}{[Z(t)+1]^{\frac{2-\delta_2}{\delta_2}}}  [Y(t)]^{\frac{2}{\delta_1}}   \leq  \ 0, \text{ for all } t > t_0.
\end{align}

Here  $ \displaystyle Z(t)=  \int_{\Omega}  [H(u)]^{p_0} dx $ , and $p_0 = {\left(\beta -\frac{q_1}{2}\right)\frac{2\delta_2}{2-\delta_2}}$.

\end{theorem}
\begin{proof}
 We follow the proof of  Theorem \ref{TH-2}  up to \eqref{No-H'-14}. Note that 
$$1 \leq \delta_2 < N \text{ and, for both cases of $\delta_1$ in \eqref{delta-2-1} and \eqref{delta-2-2},} \text{ and in both cases } \delta_1 \leq  \frac{\delta_2 N}{N-\delta_2}.$$
Then by Poincar\'e-Sobolev inequality, we have  
 \begin{align}
[Y(t)]^{\frac{1}{\delta_1}} = \left[ \int_{\Omega} \left([H(u)]^{ \gamma_1 +1  }\right)^{\delta_1} \ dx \right]^{\frac{1}{\delta_1}} 
     \leq   \ &
    c_p \left[\int_{\Omega} \bigg |  \nabla  \left([H(u)]^{\gamma_1 +1} \right)\bigg |^{\delta_2} \ dx \right]^{\frac{1}{\delta_2}}.
     \label{No-H'-15-1}
     \end{align} 
Here, $c_p > 0$ represents a constant that depends on $N$. Using the Assumption \ref{No-H'-2}, the estimate \eqref{No-H'-15-1} implies
\begin{align}
    [ Y(t)]^{\frac{\delta_2}{\delta_1}}
       \leq   \ &
        c_p^{\delta_2} c_4 \int_{\Omega}  [H(u)]^{\beta \delta_2}  \vert \nabla u\vert ^{\delta_2} \ dx. \label{No-H'-16-1}
\end{align}
By H\"older's inequality with powers $2  / (2- \delta_2)$ and $2/\delta_2$  on the right-hand-side of \eqref{No-H'-16-1} and 
 from Assumption \ref{No-H'-1}, we obtain
\begin{align}
    \int_{\Omega}  [H(u)]^{\beta \delta_2}  \vert \nabla u\vert ^{\delta_2} \ dx  
    \leq  \left[   \int_{\Omega}  [H(u)]^{\left(\beta -\frac{q_1}{2}\right)\frac{2\delta_2}{2-\delta_2}} dx\right]^{\frac{2-\delta_2}{2}}
    \left[ 
  \int_{\Omega}  [H(u)]^{q_1} \vert \nabla u\vert ^{2}  dx
    \right]^{\frac{\delta_2}{2}}. \label{No-H'-165}
\end{align}
Hence, \eqref{No-H'-16-1} becomes
\begin{align}
     [ Y(t)]^{\frac{\delta_2}{\delta_1}} \leq   c_p^{\delta_2} c_4 [Z(t)]^{\frac{2-\delta_2}{\delta_2}} \left[ 
  \int_{\Omega}  [H(u)]^{q_1} \vert \nabla u\vert ^{2}  dx
    \right]^{\frac{\delta_2}{2}}.
\end{align}

Using \eqref{No-H'-165} in \eqref{No-H'-16-1}, we obtain
\begin{align}
    [Z(t)+1]^{-\frac{2-\delta_2}{\delta_2}}c^{-2}_p c_4^ {-2/\delta_2}  [Y(t)]^{\frac{2}{\delta_1}} 
\leq   & \
  \int_{\Omega}  [H(u)]^{q_1}  \vert \nabla u\vert ^{2}  \ dx. \label{No-H'-17-1-1}
\end{align}
Combining \eqref{No-H'-14} with \eqref{No-H'-17-1-1} gives
\begin{align}
   Y'(t) 
+
 \frac{c_3 (p_1+1)}{c_4^{2/\delta_2} \cdot c_p^{2}[Z(t)+1]^{\frac{2-\delta_2}{\delta_2}}}  [Y(t)]^{\frac{2}{\delta_1}}   \leq  \ 0, \text{ for all } t>t_0.  \label{No-H'-17}
\end{align}
Hence, we derive \eqref{y-1}. 
\end{proof}

Assuming all the conditions in Theorem \ref{prop-1} hold, we present the following analogy.

\begin{corollary} \label{THM-3}
Let $Y(t), \delta_1$ and $p_0$ be defined as in Theorem \ref{prop-1}. Let $p_{*} \ge 0$ be such that $p_{*}+1 \ge p_0$. Assume that the Assumption \ref{unbb-assump} is satisfied for $p=p^*$, and
$ [H(u_0)] \in \rm L^{p_*+1}(\Omega)$.
Then $Y(t)$ satisfies the differential inequality 
\begin{align} \label{C2-ode}
      Y'(t) 
+
C_2 [Y(t)]^{\frac{2}{\delta_1}}   \leq 0,   \text{ for all } t>t_0,
\end{align}
for some constant  $ C_2 > 0$.
If $Y_{0} =0$ then  $Y(t) = 0$ for all $t\geq t_0$. 
If  $Y_{0} > 0$ then 
\begin{align} \label{as-y-1-2}
        Y(t) \leq  {\Big[  C_2\left(\frac{2-\delta_1}{\delta_1}\right)(t-t_0) +{Y_0^{-\left(\frac{2-\delta_1}{\delta_1}\right)} } \Big]^{-\frac{\delta_1}{2-\delta_1}}} \text{ for }  t \geq t_0, 
\end{align}
and consequently,
\begin{align}
    \lim_{t \rightarrow \infty} Y(t)  =0. \label{as-y-1}
\end{align}

\end{corollary}

\begin{proof}
Applying H\"older's inequality and Theorem \ref{H-bound} to $Z(t)$ in  Theorem \ref{prop-1}, gives
\begin{align}
{Z(t)}  = \int_{\Omega}  [H(u(x,t))]^{p_0} dx    
\leq &  \left[\int_{\Omega}  [H(u(x,t))]^{1+p^{*}} dx \right ]^{\frac{p_0}{1+p^{*}}}|\Omega|^{\frac{1+p^{*}-p_0}{1+p^*}} \\
\leq & 
 \left[\int_{\Omega}  [H(u_0(x)]^{1+p^{*}} dx \right ]^{\frac{p_0}{1+p^{*}}}|\Omega|^{\frac{1+p^{*}-p_0}{1+p^*}} < \infty. \label{finite-1}
\end{align}
We use the relation \eqref{finite-1} in \eqref{y-1}, leading to the inequality \eqref{C2-ode}. Consequently, leveraging Lemma \ref{Aux-lem}, the estimate \eqref{as-y-1-2}, and the limit \eqref{as-y-1} are established.
\end{proof}

\begin{remark} \label{zt-0}
Assume $ Y(t) >0
 \text{ and }   {Z(t)} >0$ for any $t\in [t_0,\infty)$. Then the estimate
 \begin{align}\label{y-1}
Y'(t) 
+
\frac{k}{[Z(t)]^{\frac{2-\delta_2}{\delta_2}}}  [Y(t)]^{\frac{2}{\delta_1}}   \leq  \ 0, \text{ for all } t > t_0.
\end{align}
holds. Next, if $p_0 \geq p_1+1$ then we will pick $ \displaystyle 
  Y(t_0)=\int [H (u(x,t_0))]^{p_0 + 1} dx$, and as a result $\displaystyle \lim_{t\to \infty} Y(t)=0$. If $Z(t_0)=0$ or $Y(t_0)=0$ then $u(x,t_0)=0$ a.e.. Moreover, because $Y(t)$ monotonically nondecreasing $u(x,t)=0$ for all $t \geq t_0$. 
\end{remark}

Our next step involves extending Assumption \ref{No-H'-2} by introducing an additional term $[H(s)]^{\beta_2}$. This modification leads to the updated assumption as follows.
\begin{assumption} \label{ext-No-H'-2}
There exist constants $\gamma_1, \beta_1, \beta_2 > 0$ and $c_5 > 0$ such that
\begin{align}
    [H^{\gamma_1+1}(s)\big]' \leq  c_5 \Big ( [H(s)]^{\beta_1} + [H(s)]^{\beta_2} \Big) \ \forall  \ s \in [0,\infty).
\end{align}
\end{assumption}

\begin{theorem} \label{stab-latest}
Let  Assumptions  \ref{No-H'-1} and \ref{ext-No-H'-2} hold with $\beta_1 > \beta_2 > q_1 \big / 2 $.
Let $\delta_1$ and $\delta_2$ be  defined by Theorem \ref{THM-3}.
Define $ \displaystyle
     Y(t) = \int_{\Omega} [H(u (x,t))]^{p_1+1} dx,  t \geq t_0.
$ 

 Then there exists  $ K > 0$  such that
\begin{align} \label{M-33}
[Y(t)]'
+
\frac{K} {   \left[Z_1(t) + Z_2(t) + 1 \right]^{\frac{2-\delta_2}{\delta_2}}} [Y(t)]^{\frac{2}{\delta_1}} \ \leq  \ 0, \text{ for all } t >t_0.  
\end{align}

Here,
\begin{align}
  {Z_1(t)} \triangleq \int_{\Omega}  [H(u(x,t))]^{p_2} dx, \text{ for } t>t_0 \text{ and }  p_2 = {\left(\beta_1 -\frac{q_1}{2}\right)\frac{2\delta_2}{2-\delta_2}} \label{z1}, \\
   {Z_2(t)} \triangleq \int_{\Omega}  [H(u(x,t))]^{p_3} dx, \text{ for } t>t_0 \text{ and }  p_3 = {\left(\beta_2 -\frac{q_1}{2}\right)\frac{2\delta_2}{2-\delta_2}}.\label{z2}
\end{align}
\end{theorem}

\begin{proof}
We follow the proof of Theorem \ref{THM-3}  up to \eqref{No-H'-15-1}. Using Assumptions \ref{No-H'-2} in \ref{No-H'-15-1};
\begin{align}
    [ Y(t)]^{\frac{\delta_2}{\delta_1}}
       \leq   \ &
        c_p^{\delta_2} c_5 \int_{\Omega}  \Big [ [H(u)]^{\beta_1 } + [H(u)]^{\beta_2 } \Big]^{\delta_2} \vert \nabla u\vert ^{\delta_2} \ dx. \nonumber\\
        \leq   \ &
        c_p^{\delta_2} c_5 2^{\delta_2} \int_{\Omega}  [H(u)]^{\beta_1\delta_2 }\vert \nabla u\vert ^{\delta_2} + [H(u)]^{\beta_2\delta_2 }\vert \nabla u\vert ^{\delta_2}  \ dx.
        \label{No-H'-15.5-2}
\end{align}
Applying \eqref{No-H'-165} for $\beta=\beta_1$ and $\beta = \beta_2$, we obtain
\begin{align}
     \int_{\Omega}  [H(u)]^{\beta_1 \delta_2}  \vert \nabla u\vert ^{\delta_2} \ dx  
     \leq & \ \left[Z_1(t) \right]^{\frac{2-\delta_2}{2}}
    \left[ 
  \int_{\Omega}  [H(u)]^{q_1} \vert \nabla u\vert ^{2}  dx
    \right]^{\frac{\delta_2}{2}} \label{B-3}, \\
     \int_{\Omega}  [H(u)]^{\beta_2 \delta_2}  \vert \nabla u\vert ^{\delta_2} \ dx  
     \leq & \ \left[Z_2(t)  \right]^{\frac{2-\delta_2}{2}}
    \left[ 
  \int_{\Omega}  [H(u)]^{q_1} \vert \nabla u\vert ^{2}  dx
    \right]^{\frac{\delta_2}{2}}.  \label{B-4}
\end{align}
Combining  \eqref{No-H'-15.5-2} with \eqref{B-3}  and \eqref{B-4} yields
\begin{align}
     [Y(t)]^{\frac{\delta_2}{\delta_1}}
     & \leq 
    c_p^{\delta_2} c_5 2^{\delta_2 + 1} \left[Z_1(t) + Z_2(t) + 1 \right]^{\frac{2-\delta_2}{2}}
    \left[ 
  \int_{\Omega}  [H(u)]^{q_1} \vert \nabla u\vert ^{2}  dx
    \right]^{\frac{\delta_2}{2}}.
    \label{B-6}
\end{align}
Raising both sides of \eqref{B-6} to the power $\frac{2}{\delta_2}$ gives
\begin{align} 
     [Y(t)]^{\frac{2}{\delta_1}}
      \leq 
      c_6 \left[Z_1(t) + Z_2(t) + 1 \right]^{\frac{2-\delta_2}{\delta_2}}
  \int_{\Omega}  [H(u)]^{q_1} \vert \nabla u\vert ^{2}  dx, \label{B-7}
\end{align}
Here $c_6 =  2^{2+ 2/\delta_2} c_p^{2} c_5^{2/\delta_2}$. Combining \eqref{No-H'-14} with \eqref{B-7} yields
\begin{align}
[Y(t)]'
+
\frac{c_3 (p+1)} { c_6  \left[Z_1(t) + Z_2(t) + 1 \right]^{\frac{2-\delta_2}{\delta_2}}} [Y(t)]^{\frac{2}{\delta_1}} \ \leq  \ 0, \text{ for all } t >t_0.  \nonumber
\end{align}
\end{proof}

Assuming all the conditions in Theorem \ref{stab-latest} hold, we present the following result.
\begin{corollary} 
Let $Y(t), \delta_1$, and $p_2$ be defined as in  Theorem \ref{stab-latest}. Let $p' \ge 0$ be such that $p'+1 \ge p_2$. Assume that the Assumption \ref{unbb-assump} is satisfied for $p=p'$, and
$ [H(u_0)] \in \rm L^{p'+1}(\Omega)$.
Then $Y(t)$ satisfies the differential inequality 
$$ \displaystyle
   Y'(t) 
+
C_1 [Y(t)]^{\frac{2}{\delta_1}}   \leq 0,   \text{ for all } t>t_0,$$
for some constant $C_1>0$. If $Y_{0} =0$ then  $Y(t) = 0$ for all $t\geq t_0$. 
If  $Y_{0} > 0$ then 
\begin{align} \label{as-y-1-3}
        Y(t) \leq  {\Big[  C_1\left(\frac{2-\delta_1}{\delta_1}\right)(t-t_0) +{Y_0^{-\left(\frac{2-\delta_1}{\delta_1}\right)} } \Big]^{-\frac{\delta_1}{2-\delta_1}}} \text{ for }  t \geq t_0, 
\end{align}
and consequently,
\begin{align}
     \lim_{t \rightarrow \infty} Y(t)  =0. \label{y-lim-5} 
\end{align} 

\end{corollary}

\begin{proof}
Recalling \eqref{z1} and \eqref{z2}, we note that $ p_{2} > p_{3}$ due to $\beta_1 > \beta_2$.  
Similarly, by applying Hölder's inequality and utilizing Theorem \ref{H-bound}, one can show that
\begin{align}
     Z_{1}(t), \ Z_2(t) < \infty , \text{ for all } t > t_0. \label{No-H'-19}
\end{align}
Hence, for some constant $C_1>0$, the differential inequality \eqref{M-33}  yields the form
$$ \displaystyle
   Y'(t) 
+
C_0 [Y(t)]^{\frac{2}{\delta_1}}   \leq 0,   \text{ for all } t>t_0,$$

Leveraging the insights from Lemma \ref{Aux-lem}, we derive the estimate \eqref{as-y-1-3} and establish the limit \eqref{y-lim-5}.
\end{proof}

\begin{remark} \label{sum-remark}
    One can consider the generalization of Assumption \ref{ext-No-H'-2}, by assuming
    \begin{align}  \label{sum-H}
         [H(s)^{\gamma_1+1}]' \leq \ \displaystyle \sum {c_i} [H(s)]^{\beta_i},
    \end{align}
   and obtain, the ODI of the form similar to \eqref{M-33}, and corresponding corollary. 
    Moreover, in analogous to Remark \eqref{zt-0}, one can obtain conclusions for $u(x,t)$  for $t \geq t_0$ under Assumption \ref{ext-No-H'-2}.
\end{remark}

\section{Acknowledgment}
We would like to express our sincere appreciation to Dr. Luan  Hoang for his immense contribution and rigorous guidance.
\bibliography{Ref}
\bibliographystyle{plain}
    \nocite{*}

\end{document}